\documentclass[reqno,a4paper,12pt]{amsart}

\def\defcolor{black!10}

\usepackage{tikz}
\usepackage{tkz-euclide}

\usetikzlibrary{decorations.pathreplacing} 

\bibliography{karamata-intersection.bib}

\usepackage[T1]{fontenc}
\usepackage[english]{babel}
\usepackage[babel]{microtype}
\usepackage[a4paper, margin=1.1in]{geometry}

\usepackage{csquotes}
\makeatletter
\newcommand{\distas}[1]{\mathbin{\overset{#1}{\kern\z@\sim}}}%
\newsavebox{\mybox}\newsavebox{\mysim}
\newcommand{\distras}[1]{%
  \savebox{\mybox}{\hbox{\kern3pt$\scriptstyle#1$\kern3pt}}%
  \savebox{\mysim}{\hbox{$\sim$}}%
  \mathbin{\overset{#1}{\kern\z@\resizebox{\wd\mybox}{\ht\mysim}{$\sim$}}}%
}
\makeatother
\usepackage{amsmath,amssymb,amsthm}

\usepackage{mathrsfs}

\usepackage{mathtools}
\usepackage{commath}

\usepackage{thmtools}
\usepackage{thm-restate}

\usepackage{bbm}
\usepackage{dsfont}
\usepackage{lmodern}
\usepackage{fixcmex}

\usepackage{cases}
\usepackage{enumitem}
\setlist[enumerate,1]{label=(\roman*)}

\usepackage{centernot}
\usepackage{booktabs}
\usepackage{multirow}
\usepackage{comment}
\usepackage{float}
\usepackage{scalerel}

\usepackage[pdftex, pdfborderstyle={/S/U/W 0}]{hyperref}
\hypersetup{
    colorlinks=true,
    linkcolor=blue!85!black,
    citecolor=blue!85!black,
    urlcolor=blue!85!black,
}

\usepackage{cleveref}

\usepackage{etoolbox}

\usepackage{subfiles}
\usepackage{pdfpages}

\usepackage{bbm}

\numberwithin{equation}{section}

\ifdefined\checkunusedreferences
\usepackage{refcheck}

\makeatletter
\newcommand{\alsocheck}[1]{%
  \expandafter\let\csname @@\string#1\endcsname#1%
  \expandafter\DeclareRobustCommand\csname relax\string#1\endcsname[1]{%
    \csname @@\string#1\endcsname{##1}\@for\@temp:=##1\do{\wrtusdrf{\@temp}\wrtusdrf{{\@temp}}}}%
  \expandafter\let\expandafter#1\csname relax\string#1\endcsname
}
\newcommand{\alsocheckrange}[1]{%
  \expandafter\let\csname @@\string#1\endcsname#1%
  \expandafter\DeclareRobustCommand\csname relax\string#1\endcsname[2]{%
    \csname @@\string#1\endcsname{##1}{##2}\wrtusdrf{##1}\wrtusdrf{{##1}}\wrtusdrf{##2}\wrtusdrf{{##2}}}%
  \expandafter\let\expandafter#1\csname relax\string#1\endcsname
}
\makeatother

\alsocheck{\cref}
\alsocheck{\Cref}
\alsocheckrange{\crefrange}
\alsocheckrange{\Crefrange}

\else
\fi

\ifdefined\thmcolor
\declaretheoremstyle[
  shaded={bgcolor=\thmcolor,
  }
]{plain}
\else
\fi

\ifdefined\defcolor
\declaretheoremstyle[
  headfont=\normalfont\bfseries,
  bodyfont=\normalfont
]{noital}
\else
\declaretheoremstyle[
  headfont=\normalfont\bfseries,
  bodyfont=\normalfont,
]{noital}
\fi

\declaretheorem[style=plain,numberwithin=section,name=Theorem]{theorem}

\declaretheorem[style=plain,sibling=theorem,name=Proposition]{proposition}

\declaretheorem[style=plain,sibling=theorem,name=Lemma]{lemma}

\declaretheorem[style=plain,sibling=theorem,name=Corollary]{corollary}

\declaretheorem[style=plain,numbered=no,name=Theorem]{theoremn-n}
\declaretheorem[style=plain,numbered=no,name=Theorems]{theorems-n}
\declaretheorem[style=plain,numbered=no,name=Proposition]{proposition-n}
\declaretheorem[style=plain,numbered=no,name=Propositions]{propositions-n}
\declaretheorem[style=plain,numbered=no,name=Lemma]{lemma-n}
\declaretheorem[style=plain,numbered=no,name=Lemmas]{lemmas-n}
\declaretheorem[style=plain,numbered=no,name=Corollary]{corollary-n}
\declaretheorem[style=plain,numbered=no,name=Corollaries]{corollaries-n}
\declaretheorem[style=plain,numbered=no,name=Conjecture]{conjecture-n}
\declaretheorem[style=plain,numbered=no,name=Conjectures]{conjectures-n}
\declaretheorem[style=plain,numbered=no,name=Question]{question-n}
\declaretheorem[style=plain,numbered=no,name=Questions]{questions-n}
\declaretheorem[style=plain,numbered=no,name=Claim]{claim-n}
\declaretheorem[style=plain,numbered=no,name=Claims]{claims-n}
\declaretheorem[style=plain,numbered=no,name=Fact]{fact-n}
\declaretheorem[style=plain,numbered=no,name=Facts]{facts-n}
\declaretheorem[style=plain,numbered=no,name=Problem]{problem-n}
\declaretheorem[style=plain,numbered=no,name=Problems]{problems-n}
\declaretheorem[style=plain,numbered=no,name=Open Problem]{openproblem-n}
\declaretheorem[style=plain,numbered=no,name=Open Problems]{openproblems-n}
\declaretheorem[style=plain,numbered=no,name=Challenge]{challenge-n}
\declaretheorem[style=plain,numbered=no,name=Challenges]{challenges-n}
\declaretheorem[style=plain,numbered=no,name=Exercise]{exercise-n}
\declaretheorem[style=plain,numbered=no,name=Exercises]{exercises-n}
\declaretheorem[style=plain,numbered=no,name=Property]{property-n}
\declaretheorem[style=plain,numbered=no,name=Properties]{properties-n}

\declaretheorem[style=noital,sibling=theorem,name=Remark]{remark}

\declaretheorem[style=noital,sibling=theorem,name=Definition]{definition}

\declaretheorem[style=noital,numbered=no,name=Remark]{remark-n}
\declaretheorem[style=noital,numbered=no,name=Remarks]{remarks-n}
\declaretheorem[style=noital,numbered=no,name=Definition]{definition-n}
\declaretheorem[style=noital,numbered=no,name=Definitions]{definitions-n}
\declaretheorem[style=noital,numbered=no,name=Construction]{construction-n}
\declaretheorem[style=noital,numbered=no,name=Constructions]{constructions-n}
\declaretheorem[style=noital,numbered=no,name=Observation]{observation-n}
\declaretheorem[style=noital,numbered=no,name=Observations]{observations-n}
\declaretheorem[style=noital,numbered=no,name=Example]{example-n}
\declaretheorem[style=noital,numbered=no,name=Examples]{examples-n}

\renewcommand{\le}{\leqslant}
\renewcommand{\leq}{\leqslant}
\renewcommand{\ge}{\geqslant}

\let\oldexists\exists
\let\exists\relax
\DeclareMathOperator{\exists}{\:\!\oldexists}
\let\oldforall\forall
\let\forall\relax
\DeclareMathOperator{\forall}{\:\!\oldforall}

\undef{\set}
\DeclarePairedDelimiter{\set}{\lbrace}{\rbrace}

\undef{\emptyset}
\newcommand{\emptyset}{\varnothing}

\DeclareMathOperator{\ind}{\mathbf{1}}

\renewcommand{\d}{\mathop{}\!\mathrm{d}}

\newcommand{\du}{\d u}

\undef{\mod}
\newcommand{\mod}[1]{\ (\mathrm{mod}\ #1)}

\undef{\abs}
\DeclarePairedDelimiterX{\abs}[1]
  {\lvert}{\rvert}{\ifblank{#1}{\,\cdot\,}{#1}}
\undef{\norm}
\DeclarePairedDelimiterX{\norm}[1]
  {\lVert}{\rVert}{\ifblank{#1}{\,\cdot\,}{#1}}
\DeclarePairedDelimiterX{\inner}[2]
  {\langle}{\rangle}{\ifblank{#1}{\,\cdot\,}{#1},\ifblank{#2}{\,\cdot\,}{#2}}

\DeclareMathDelimiter{\given}
  {\mathbin}{symbols}{"6A}{largesymbols}{"0C}
\DeclareMathOperator{\Prob}{\mathbb{P}}
\DeclarePairedDelimiterXPP{\prob}[1]
  {\Prob}{\lparen}{\rparen}{}
  {\renewcommand{\given}{\nonscript\;\delimsize\vert\nonscript\;\mathopen{}}#1}
\DeclareMathOperator{\Expec}{\mathbb{E}}
\DeclarePairedDelimiterXPP{\expec}[1]
  {\Expec}{\lparen}{\rparen}{}
  {\renewcommand{\given}{\nonscript\;\delimsize\vert\nonscript\;\mathopen{}}#1}
\DeclareMathOperator{\Var}{Var}
\DeclarePairedDelimiterXPP{\var}[1]
  {\Var}{\lparen}{\rparen}{}
  {\renewcommand{\given}{\nonscript\;\delimsize\vert\nonscript\;\mathopen{}}#1}
\DeclareMathOperator{\Cov}{Cov}
\DeclarePairedDelimiterXPP{\cov}[2]
  {\Cov}{\lparen}{\rparen}{}{#1,#2}

\let\SS\relax

\newcommand{\PP}{\mathbb{P}}

\newcommand{\RR}{\mathbb{R}}
\newcommand{\SS}{\mathbb{S}}

\DeclareMathOperator{\Li}{Li}
\DeclareMathOperator{\aff}{aff}

\begin{document}

\title[random chords on a circle]{On extremal problems associated with random chords on a circle}

\author[C. Bortolotto]{Cynthia Bortolotto}
\address{Department of Mathematics, ETH Zürich, Zürich, Switzerland}
\email{cynthia.bortolotto@math.ethz.ch}

\author[J. P. G. Ramos]{Jo\~{a}o P. G. Ramos}
\address{Institute of Mathematics, EPFL, Lausanne, Switzerland}
\email{joao.ramos@epfl.ch}


\begin{abstract}
Inspired by the work of Karamata, we consider an extremization problem associated with the probability of intersecting two random chords inside a circle of radius $r, \, r \in (0,1]$, where the endpoints of the chords are drawn according to a given probability distribution on $\mathbb{S}^1$.

We show that, for $r=1,$ the problem is degenerated in the sense that any \emph{continuous} measure is an extremiser, and that, for $r$ sufficiently close to $1,$ the desired maximal value is strictly below the one for $r=1$ by a polynomial factor in $1-r.$ Finally, we prove, by considering the auxiliary problem of drawing a single random chord, that the desired maximum is $1/4$ for $r \in (0,1/2).$ Connections with other variational problems and energy minimization problems are also presented. 
\end{abstract}

\maketitle


\section{Introduction}

\subsection{Historic background and main results} Let $C=\mathbb{S}^1 \subset \RR^2$ be the unit circle and let $\Pi_n$ be a regular $n$-sided polygon inscribed in $C$.  Consider all the diagonals of $\Pi_n$. A natural question in incidence geometry concerns the distribution of the intersection points of these diagonals.

In more effective terms, let $C_r$ be the circle of radius $r$, concentric with $C$ and let $I_n(r)$ be the number of intersection of these diagonals - counted with multiplicity -inside $C_r$. In 1961, Karamata \cite{Karamata} proved that
$$\lim_{n \rightarrow \infty} \frac{I_n(r)}{I_n(1)} = \frac2{\pi^2} \Li_2(r^2), $$
for any $0 \le r \le 1,$ where $\Li_2$ is the Euler's dilogarithm function defined by
$$ \Li_2(x) = \sum_{k=0}^\infty \frac{x^k}{k^2}, $$
for $x \in \mathbb{C}\setminus 0\le x \le 1$. Since the work of Karamata, that problem remained relatively dorment until recently, when the first author and V. Souza \cite{Bortolotto-Souza} extended the previous results to diagonals intersecting outside the circle, through both analytic and geometric methods, simplifying and extending Karamata's proof. 

As a matter of fact, one of the essential reductions of \cite{Bortolotto-Souza} uses that the problem can be reframed in a continuous setting: consider four independent random points $X_1$, $X_2$, $X_3$, $X_4$, chosen uniformly from the unit circle $\SS^1 \subseteq \RR^2$ centered at the origin $O \in \RR^2$.
Let $\aff(x,y)$ be the affine subspace spanned by $x,y \in \RR^2,$ whenever $x\neq y$ (observe that if $x=y$, the definition implies that $\aff(x,y)=x).$
With full probability $\aff(X_1,X_2)$ is one dimensional and intersect $\aff(X_3,X_4)$ at a single point. A main result of \cite{Bortolotto-Souza} is that 
\begin{equation*}
    \prob[\big]{||\aff(X_1,X_2) \cap \aff(X_3,X_4)||_2 < r} = \frac2{\pi^2} \Li_2(r^2),
\end{equation*}
for $0\leq r \leq 1$. 

In this manuscript, we shall be concerned with an optimisation problem inspired by the results in \cite{Karamata} and \cite{Bortolotto-Souza}. In particular, we are interested in knowing, for fixed $r \in (0,+\infty),$ which probability distribution in the unit circle maximises 
$$\mathbb{P}_{\mu} \left(||\aff(X_1,X_2) \cap \aff(X_3,X_4)||_2 < r\right), $$
where $X_i$ are independent and drawn according to $\mu$ on $\mathbb{S}^1.$ Here, we let, in case $x_1 = x_2,$ $\aff(x_1,x_2) = \aff(x_1,x_1)$ denote the tangent space to $\mathbb{S}^1$ at $x_1.$ 

In this regard, note first that, for $r> 1,$ the desired maximum trivially equals $1$, which follows by considering $\mu = \delta_1.$ Hence, in what follows we will concentrate on the case $r \in [0,1].$ We introduce first some notation.

Indeed, let us define a function $f: (\SS^1)^4 \rightarrow [0,\infty]$ by
\begin{align*}
    f(x_1,x_2,x_3,x_4)=\begin{cases}
||\aff(x_1,x_2) \cap \aff(x_3,x_4)||_2 & \text{ if } \aff(x_1,x_2)\cap\aff(x_3,x_4)=\{x_0\} \\
 \operatorname{d}( \aff(x_1,x_2), 0) & \text{ if } \aff(x_1,x_2) = \aff(x_3,x_4) \\
\infty & \text{ if } \aff(x_1,x_2)\cap\aff(x_3,x_4) = \emptyset,
    \end{cases}
\end{align*}
where $\operatorname{d}( \aff(x_1,x_2), 0) $ denotes the Euclidean distance of the line to the origin. We denote by $\ell = f(X_1,X_2,X_3,X_4)$ and if $X_i$, $i=1,2,3,4$, are $\mu$ distributed and we are interested in studying
\begin{align*}
    \PP_\mu(\ell < r  ) := \int_{(\mathbb{S}^1)^4} \ind_{ \{ f(x_1,x_2,x_3,x_4)<r \}} \d\mu(x_1) \d\mu(x_2)\d\mu(x_3)\d\mu(x_4).
\end{align*}

We consider its maximal version.

\begin{definition}
   For $0\leq r\leq 1$ and $\mu \in \mathcal{M}(\mathbb{S}^1)$, let
\begin{align}
    A(r)= \sup_{\mu \in \mathcal{M}(\mathbb{S}^1)} \PP_{\mu}(\ell <r)
\end{align}
and
\begin{align}
        \overline{A}(r)= \sup_{\mu \in \mathcal{M}(\mathbb{S}^1)} \PP_\mu(\ell \le r).
\end{align}
\end{definition}
We are hence interested in understanding the behaviour of the functions $A$ and $\overline{A}$. The first main result of this manuscript a characterization of what happens at the $r=1$ case. Contrarily to what happens in many problems in geometric-flavoured optimisation, the space of extremals is degenerate - indeed, the optimisers for that problem are \emph{exactly} the continuous measures on the circle. Effectively, we will show the following: 

\begin{theorem}\label{thm:boundary} We have that
\begin{align*}
    A(1) = \frac{1}{3}.
\end{align*}
Moreover, $\mu$ is a measure such that $\mathbb{P}_{\mu}(\ell < 1) = \frac13$ if, and only if, $\mu(\{p\}) = 0$ for each $p \in \mathbb{S}^1.$  
\end{theorem} 

In order to prove Theorem \ref{thm:boundary}, we will first reduce to the case of \emph{discrete} measures. In that case, the strict inequality 
\[
\mathbb{P}_{\mu}(\ell < 1) < \frac13
\]
holds. We then use a discretization procedure to show that the non-strict inequality holds for all measures. Finally, we prove that, if a measure gives non-trivial weight to any single point in $\mathbb{S}^1,$ then it is not an extremal. Since all continuous measures are easily seen to be extremal, this concludes the proof of that result. 

As a consequence of the proof of Theorem \ref{thm:boundary}, we are able to show a weak stability version of Theorem \ref{thm:boundary}: any measure $\mu$ close to satisfying $\mathbb{P}_{\mu}(\ell < 1) = \frac13$ must, indeed, give small masses to each individual point.

Our next result, which is directly related to Theorem \ref{thm:boundary}, is a description of what happens near $t=1:$ 

\begin{theorem}\label{thm:almot-boundary} Let $\mu$ be a measure on $\mathbb{S}^1.$ Then, if $\theta > 0$ is sufficiently small, we have 
\[
\mathbb{P}_{\mu}(\ell < 1-\theta) \le \frac13 - C \theta^2.
\]
\end{theorem}
This result can be understood as a sort of ``uncertainty principle'' for the distribution of intersections of chords. That is, we cannot perfectly concentrate the intersections of chords inside any circle slightly smaller than the full circle; some amount of the intersections must lie on the remaining annulus. 

The proof of Theorem \ref{thm:almot-boundary} is through carefully analysing boundary behaviour: indeed, the idea is to look at the continuous part of the measure $\mu.$ If it has large mass, then it is simple to show that its intersections have to be `spread out'. On the other hand, if the continuous part is small, we either have that a single point mass is not too small - in which case we violate the corollary of Theorem \ref{thm:boundary} mentioned above -, or all point masses are small. If the latter is the case, then the measure roughly looks like a continuous one, and thus a suitable version of the argument employed for the proof of the continuous part can be used, finishing the argument. 

As a by-product of Theorem \ref{thm:almot-boundary}, we have the following result: 

\begin{corollary} Let $\theta > 0$ be sufficiently small. Then there are absolute constants $c_1,c_2>0$ such that we have 
\[
\frac13 - c_1 \frac{\theta}{\log (1/\theta)} \le A(1-\theta) \le \frac13 - c_2 \theta^2.
\]
\end{corollary} 

The lower bound in such a result comes from the explicit example by Karamata \cite{Karamata}. We would be inclined to believe that the lower bound is the true behaviour of the function $A$ near the boundary, in spite of not being able to prove that at the moment. 

Finally, the last main result of this manuscript is an explicit computation of $A(r)$ for $r \in [0,1/2].$ We first point out that $A(r) \ge \frac14,$ for every $r \in (0,1).$ To see that, one just have to consider the measure $\mu = \frac12\delta_{ \{x\} } + \frac12 \delta_{ \{ -x \} },$ for any $x\in \mathbb{S}^1.$ The result below shows that we cannot do better than that example. 

\begin{theorem}
\label{rsmaller12}
    For any $r\in [0,\frac12]$ it holds that 
    $$A(r) = \frac14.$$
\end{theorem}

For the proof of that result, we need to analyse a slightly simpler problem: given $\mu,$ draw two independent random variables $X,Y$ according to $\mu$ on $\mathbb{S}^1.$ For which $\mu$ is the number 
\[
\mathbb{P}_{\mu}(\d(\aff(X,Y),0) < r)
\]
the largest? Let $B(r)$ denote such a maximum - that is, 
\[
B(r) = \sup_{\mu \in \mathcal{M}(\mathbb{S}^1)} \mathbb{P}_{\mu}(\d(\aff(X,Y),0) < r). 
\]
Clearly, we have that $A(r) \le B(r)^2$ pointwise. Since it is easy to show that $B(r) \to 1$ as $r \to 1,$ this inequality produces bad upper bounds for $r$ not small, but the intuition is that, for $r$ sufficiently small, two chords intersecting and the two of them intersecting a small circle should coincide in the extreme cases. We then concentrate on showing that $B(r) = \frac12$ for $r \in [0,1/2].$

For that task, we switch to a slightly modified version of the problem, in which we are able to prove that extremal measures actually exist. For such measures, extremality implies a geometric property of its support - basically, the support has to give lots of mass to sets which, for $r<\frac12,$ are disjoint. The result then follows by the fact that, if we had $B(r) > \frac12,$ we would have the two disjoint parts of the support of $\mu$ have mass  strictly larger than $\frac{1}{2}$ each, a contradiction to the fact that $\mu$ is a probability measure. 

An interesting feature of the numbers $B(r)$ is that they may be represented as suitable \emph{energy minimisation} problems. Contrarily to most of the problems of such flavor, for which one needs an orthogonal basis expansion with nice positivity/negativity properties, the nature of our approach is purely geometric, which might be of help in the study of further problems of such kind in the future. 

This article is structured as follows: in Section \ref{sec:boundary-behavior}, we prove Theorems \ref{thm:boundary} and \ref{thm:almot-boundary}. Then, in Section \ref{sec:small-r}, we study the problem of determining $B(r)$, proving Theorem \ref{rsmaller12}, as well as explicitly computing $B(r)$ for an interval slightly larger than $(0,1/2)$. Finally, in Section \ref{sec:general}, we discuss general facts about the problem of finding $A(r)$ and measures $\mu$ achieving such a supremum, and we additionally provide a classical argument for the computation of the maximal/minimal expected value of the distance of $\aff(X,Y)$ to $0,$ where $X,Y \distas d \mu$ are independent random variables.

\vspace{2mm}

\subsection{Notation} We will write $A \gtrsim B$ if there is an absolute constant $C>0$ such that $A \ge C \cdot B.$ In general, we shall write $\mathbb{P}_{\mu}(E) $ for $\int_{E} \, \d \mu.$ Finally, for an arc $I \subset \mathbb{S}^1$ we shall write $I = (e^{ia},e^{ib})$ to denote the counter-clockwise oriented arc with endpoints $e^{ia}$ and $e^{ib}.$  

\section{Behaviour near $r=1$}\label{sec:boundary-behavior}

\subsection{Discrete measures.}

Let $\mu$ be a discrete measure of the form $\mu = \sum_{i=1}^n \alpha_i \ind_{a_i}$, where $a_i\neq a_j$ for $i\neq j$, $0< \alpha_i < 1$, $n\ge 4$ and $\sum_{i=1}^n\alpha_i =1.$  

\begin{lemma}
\label{lemma1}
    For $\mu$ as defined above and $n\ge 4$, it holds that
    \begin{align*}
        \PP_{\mu} (\ell < 1 ) = \frac{1}{3} - 2\sum_{i=1}^n \alpha_i^2 + \frac83 \sum_{i=1}^n\alpha_i^3 + 3\left( \sum_{i=1}^n \alpha_i^2 \right)^2 - 4\sum_{i=1}^n \alpha_i^4.
    \end{align*}
    If $n=2$, then
    \begin{align*}
        \PP_{\mu} (\ell < 1 ) = 4\alpha_1^2\alpha_2^2
    \end{align*}
    and for $n=3$ it holds that
    \begin{align*}
        \PP_{\mu} (\ell < 1 ) = 4(\alpha_1^2\alpha_2^2+\alpha_1^2\alpha_3^2+\alpha_2^2\alpha_3^2).
    \end{align*}
\end{lemma}

\begin{proof}
     In the cases $n=2$ and $n=3$, the only intersections possible occur when one selects twice the same chord. An analogous observation proves the result in those cases. Hence, we focus on the case in which $n \ge 4.$

    For $n\ge 4$, the intersection of chords inside the circle are given by either selecting four different $a_i$, ordered so that the intersection lies inside the circle, or either by selecting twice the same chord. This yields the equality

    \begin{align*}
       \PP_{\mu} (\ell < 1 ) = 8 \sum_{i=1}^n\alpha_i \sum_{j>i+1}\alpha_j\sum_{i<l<j}\alpha_l\sum_{k>j}\alpha_k + 4\sum_{i=1}^n\alpha_i^2\sum_{j>i}\alpha_j^2.
    \end{align*}
Using the Girard–Newton formulae, we can write

\begin{equation}
\begin{split}
\label{newton1}
 8 \sum_{i=1}^n\alpha_i \sum_{j>i+1}\alpha_j\sum_{i<l<j}\alpha_l\sum_{k>j}\alpha_k & = \frac13 \left(\sum_{i=1}^n\alpha_i  \right)^4 - 2\left(\sum_{i=1}^n\alpha_i \right)^2 \sum_{i=1}^n\alpha_i^2 \\
 & + \frac83 \sum_{i=1}^n\alpha_i \sum_{i=1}^n\alpha_i^3 + \left(\sum_{i=1}^n\alpha_i^2 \right)^2 - 2\sum_{i=1}^n\alpha_i^4 \\
 & =\frac13 - 2\sum_{i=1}^n\alpha_i^2 +\frac83 \sum_{i=1}^n\alpha_i^3 + \left(\sum_{i=1}^n\alpha_i^2 \right)^2 - 2\sum_{i=1}^n\alpha_i^4 ,
 \end{split}
\end{equation}

and
\begin{equation}
\begin{split}
\label{newton2}
4\sum_{i=1}^n\alpha_i^2\sum_{j>i}\alpha_j^2 & = 2\left(\sum_{i=1}^n \alpha_i^2 \right)^2 - 2\sum_{i=1}^n \alpha_i^4.
\end{split}
\end{equation}
The two equalities (\ref{newton1}) and (\ref{newton2}) imply the result.

\end{proof}

\begin{proposition}
\label{prop1}
    For any $\mu$ as above, it holds that
    \begin{align*}
        \PP_\mu (\ell < 1 ) < \frac{1}{3}.
    \end{align*}
\end{proposition}
\begin{proof}
We first prove that the result holds for $n=2,3.$ Indeed, for $n=2$, we have
\begin{align*}
    4x^2(1-x)^2
\end{align*}
for $x \in (0,1)$, which has a maxima of $\frac14$ for $x=\frac12.$ For $n=3$ we have
\begin{align*}
    4(\alpha_1^2\alpha_2^2 + \alpha_1^2\alpha_3^2 + \alpha_2^2\alpha_3^2) < \frac13.
\end{align*}

Observe that it suffices to prove that

\begin{align}
\label{mainineq}
    2\sum_{i=1}^n \alpha_i^2 - \frac83 \sum_{i=1}^n\alpha_i^3 - 3\left( \sum_{i=1}^n \alpha_i^2 \right)^2 + 4\sum_{i=1}^n \alpha_i^4 > 0 
\end{align}
for $\alpha_i$ as described above. As a matter of fact, we shall prove the following improved version below: 

\begin{lemma}\label{lemma:improved-lower} Let $D = \{ (\alpha_i)_{i=1}^n \in \RR^n : \sum_{i=1}^n\alpha_i=1, 0<\alpha_i<1 \}$ and $w: D \rightarrow \RR$ defined by,
\begin{align*}
    w(\alpha_1, \ldots, \alpha_n) =     2\sum_{i=1}^n \alpha_i^2 - \frac83 \sum_{i=1}^n\alpha_i^3 - 3\left( \sum_{i=1}^n \alpha_i^2 \right)^2 + 4\sum_{i=1}^n \alpha_i^4 .
\end{align*} 
Then there exists a numerical constant $C>0$ such that 
\begin{equation*}
w(\alpha_1,\dots,\alpha_d) \ge C \left( \sum_{i=1}^n \alpha_i^2 \right),
\end{equation*}
for all $(\alpha_i)_{i=1}^n \in D.$
\end{lemma}

\begin{proof}
We will prove the result with $C = 0.01,$ but this can easily be improved by carefully redoing the proof below. 

Denote by $S=\sum_{i=1}^n\alpha_i^2$. We first assume that $S<\frac{14}{27}-0.01$.
Observe that $4\alpha_i^2-\frac{8}{3}\alpha_i + 2$ attains its minimum of $\frac{14}{9}$ when $\alpha_i=\frac13$. Thus,
\begin{align}
\label{ineq1}
    \sum_{i=1}^n \alpha_i^2 (4\alpha_i^2-\frac83 \alpha_i + 2 ) \ge \frac{14}{9}\sum_{i=1}^n \alpha_i^2 > 3\left(\sum_{i=1}^n\alpha_i^2 \right)^2 + 0.01 \cdot S,
\end{align}
where the last inequality holds since $S<\frac{14}{27} -0.01$. So, for the rest of the proof we may assume that the minimum value of $w$ given by $(\alpha_i)_{i=1}^n$ such that $S\ge \frac{14}{27} - 0.01.$

Writing the Lagrange multiplier equation for $w$ we conclude that $\alpha_i$, for $i=1,\ldots, n$, must satisfy
\begin{align}
\label{eqraizes}
    16\alpha_i^3-8\alpha_i^2 + (4-12S-0.02)\alpha_i +\lambda=0,
\end{align}
where $\lambda \in \RR$ is a parameter.

Furthermore, we observe that the cubic $(\ref{eqraizes})$ has at most two solutions $\gamma, \beta$ in the interval $(0,1).$ Indeed, if the equation had three positive solution then it would hold that $\frac{1}{16}(4-12S-0.01)>0$, which is a contradiction since $S\ge\frac{14}{27}-0.01>\frac{1}{3}.$

In order to understand where the positive roots lie, we compute the derivative of the equation (\ref{eqraizes}) in the $\alpha_i$ variable: 
\begin{align}
12x^2-4x+1-3S-0.005=0
\end{align}
and observe that its roots are given by
\begin{align}
    x = \frac{1}{6}\left( 1 \pm \sqrt{9S + 0.24 -2} \right).
\end{align}
Using that $S\ge \frac{14}{27}-0.01$ we can conclude that one of the roots satisfies $\gamma > 0.42.$ We just need to consider three cases then:

\vspace{2mm}

\textbf{Case 1:} $\gamma$ does not appear among $(\alpha_i)_{i=1}^n$. Then $\alpha_i=\frac1{n}$ for $i=1,\ldots,n$ and inequality (\ref{ineq1}) becomes
\begin{align*}
\frac{1.99}{n} - \frac83 \cdot\frac1{n^2} - 3\frac1{n^2} + 4\frac{1}{n^3}>0,
\end{align*}
which is true for $n\ge 2.$

\vspace{2mm}

\textbf{Case 2:} $\gamma$ appears once among $(\alpha_i)_{i=1}^n$. In this case, we want to prove that
\begin{align}
\label{ineq2}
    \gamma^2\left(4\gamma^2-\frac83 \gamma +1.99\right) + (n-1)\beta^2\left(4\beta^2-\frac83 \beta+1.99\right) > 3( (n-1)\beta^2 + \gamma^2)^2
\end{align}
where $\beta =\beta(\gamma)= \frac{1-\gamma}{n-1}$. 

Let $g(x)=x^2(4x^2-\frac83x+1.99)$. Inequality (\ref{ineq2}) translates to
\begin{align}
    g(\gamma) + (n-1)g(\beta) > 3( (n-1)\beta^2 + \gamma^2)^2.
\end{align}
The case $n=2$ simplifies to
\begin{align*}
\frac{97}{300} - 4 \gamma^4 + 8 \gamma^3 - \frac{201}{50} \gamma^2 + \frac{\gamma}{50}>0,
\end{align*}
and this inequality can easily be shown to hold for $\gamma \in (0.42,1).$ We suppose hence from now on that $n\ge3.$ For that case, we shall relabel variables as follows: let $k = n-1, 1-\gamma = \theta.$ Then (\ref{ineq2}) is equivalent to 
\begin{align}
 \gamma^2\left(4\gamma^2-\frac83 \gamma +1.99\right) + \frac{\theta^2}{k} \left(4\frac{\theta^2}{k^2}-\frac83 \frac{\theta}{k}+1.99\right) > 3\left( \frac{\theta^2}{k} + \gamma^2\right)^2.
\end{align}
Again, we change our point of view by letting $\frac{1}{k} = t,$ and considering the resulting polynomial in $t.$ The inequality we want to prove then rewrites as
\begin{equation}\label{eq:poly-t}
    p_{\gamma}(t) = 4 \theta^4 t^3 -\left(\frac{8}{3} \theta^3 + 3 \theta^4 \right)\cdot  t^2 + (1.99 \theta^2  -6\theta^2\gamma^2) t -3\gamma^4 + g(\gamma) > 0,
\end{equation}
where $t$ is supposed to be of the form $1/k, k \ge 2, \, k \in \mathbb{N}.$ We will show slightly more, by showing that the function defined in (\ref{eq:poly-t}) is, as a matter of fact, \emph{concave} for $t \in (0,1/2),$ and that the inequality indeed holds at the endpoints. 

As a matter of fact, for $t=0,$ we obtain the inequality 
$$\gamma^4 - \frac83 \gamma^3 + 1.99 \gamma^2 > 0,$$
which is easily seen to be true for $\gamma \in (0.42,1).$ Similarly, for $t = 1/2,$ the associated polynomial equation we obtain is 
\[
-\frac{9}{4}\gamma^4 + 5\gamma^3 -\frac{703}{200}\gamma^2 +\frac{101}{100}\gamma + \frac{47}{60} > 0,
\]
which may be once more directly verified for $\gamma \in (0.42,1).$ Finally, notice that 
$$p_{\gamma}''(t) = 24 \theta^4 t - 2 \cdot \left( \frac{8}{3} \theta^3 + 3\theta^4\right),$$
which is negative as long as we have $t < \frac{2}{9 \theta} + \frac{1}{4}.$ As $\theta = 1-\gamma < 0.58,$ the right-hand side of that upper bound on $t$ is at least $0.63,$ which is larger than $1/2,$ and hence $p_{\gamma}$ is indeed concave for $t$ between $0$ and $1/2,$ which concludes the desired assertion. 




\vspace{2mm}

\textbf{Case 3:} $\gamma$ appears twice among $(\alpha_i)_{i=1}^n$. In this case, it holds that $\delta=\delta(\gamma)=\frac{1-2\gamma}{n-2}$ and that $0.42<\gamma\le0.5.$ The hypothesis $S \ge \frac{14}{27}-0.01$ implies that
\begin{align*}
    2\gamma^2 &+ \frac{(1-2\gamma)^2}{n-2}  \ge \frac{14}{27}-0.01 \\
\Rightarrow     n-2 &\le \frac{(1-2\gamma)^2}{\frac{14}{27}-0.01-2\gamma^2} \le 0.165,
\end{align*}
which implies that $n=2.$ In this case we have $\gamma = 1/2,$ but since $2\gamma^2 = 1/2 < \frac{14}{27} - 0.01,$ this case is empty, finishing our analysis, and completing the proof of Lemma \ref{lemma:improved-lower}. 
\end{proof}

Clearly, the proposition follows at once from Lemma \ref{lemma:improved-lower}, and we are done. 
\end{proof}

\begin{remark} The value $C = 0.01$ of the constant given by the conclusion of Lemma \ref{lemma:improved-lower} is likely very far from optimal, as a quick scan through the analysis above shows. We would, nonetheless, expect that essentially the same proof, together with some additional carefully chosen estimates, should be able to yield the sharp constant in that result. 
\end{remark}

\subsection{Continuous measures are (the only) optimal ones}

By the Lebesgue decomposition theorem, any measure $\mu$ can be written as $\mu = \mu_c + \mu_d,$ where $\mu_c$ is the continuous part, that is, $\mu_c(\{x\}) = 0$ for any $x\in \SS^1$ and $\mu_d$ is the discrete part. We prove first that if $\mu$ has no discrete part then it is an extremiser of $A(1).$

\begin{proposition}\label{prop:continuous-extreme}
    Let $\mu$ be a measure with no discrete part, that is, $\mu_d \equiv 0.$ Then
    $$ \PP_{\mu} (\ell < 1 ) = \frac13.$$
\end{proposition}

\begin{proof}
    Let $X_1, X_2, X_3, X_4$ be i.i.d random variables with density given by $\mu$ and $I_1, I_2, I_3, I_4 \in \SS^1$ disjoint intervals. Denote by $B$ the event
    \begin{align*}
        B=\{ X_1 \in I_{\sigma(1)}, X_2 \in I_{\sigma(2)}, X_3 \in I_{\sigma(3)}, X_4 \in I_{\sigma(4)}, \text{ for any }\sigma \text{ permutation} \},
    \end{align*}
    and observe that $\mathbb{P}(B) = \sum_{\sigma} \mu(I_1)\mu(I_2)\mu(I_3)\mu(I_4) = 24\mu(I_1)\mu(I_2)\mu(I_3)\mu(I_4).$
On the other hand, it holds that
\begin{align*}
    \mathbb{P}(\{ \ell < 1\} \cap B) =8\mu(I_1)\mu(I_2)\mu(I_3)\mu(I_4),
\end{align*}
so we conclude that
    \begin{align*}
        \PP_\mu(\ell <1 | B) = \frac{1}{3}.
    \end{align*}
Given any positive integer $k$ we take $\Lambda_k$ to be a finite covering of disjoint intervals of the form $(a,b]$ of $\SS^1$ satisfying the following condition: for any interval $I \in \Lambda_k$ it holds that $\frac{1}{2k} < \mu(I) < \frac{2}{k}$.
Since the measure has no discrete part, such a covering always exists.
Let $A_k$ be the event $A_k= (X_1, X_2, X_3, X_4
\text{ belong to different intervals of }\Lambda_k )$ and observe that it satisfies 
\begin{align*}
    \mu(A_k) \ge 1-6\sum_{I \in \Lambda_k}\mathbb{P}(X_1, X_2 \in I) \ge 1 - \frac{48}{k}.
\end{align*}

Using this property, we can write 
\begin{align*}
    \PP_\mu(\ell < 1 ) = \PP_\mu(\ell < 1 \cap A_k ) + \PP_\mu(\ell < 1 \cap A_k^c) =\frac13 \cdot \mu(A_k) + \PP_\mu(\ell < 1 \cap A_k^c)
\end{align*}
and using that $\mu(A_k)$ converges to $1$ as $k$ tends to infinity, we conclude that $\PP_\mu(\ell < 1 ) = \frac13.$
 \end{proof}
We are now able to prove the converse, that is, that the only measures satisfying equality are indeed purely continuous ones. 

\begin{proposition}\label{prop:only-extreme-continuous}
    Let $\mu = \sum_{i=1}^n \alpha_i\delta_{\{x_i\}} + \mu_c$ be a measure in $\SS^1,$ with $\mu_c$ its continuous part satisfying $\mu_c(\SS^1) = p, 0<p<1.$  Then
    \begin{align*}
        \PP_\mu(\ell < 1) < \frac13.
    \end{align*}
\end{proposition}

\begin{proof}
    For every positive integer $k$, we take a covering $\Lambda_k$ of $\SS^1 \backslash \{x_1, \ldots, x_n\}$ consisting of disjoint intervals $I_j^k$ satisfying $ \frac{1}{2k} < \mu(I_j^k) < \frac{2}{k}.$ For every $I_j^k$ we select a point $y_j^k \in I_j^k$ and define the discrete measure
    \begin{align*}
        \mu_k = \sum_{i=1}^n\alpha_i \delta_{ \{x_i\}} + \sum_{j=1}^l \mu(I_j^k) \delta_{ \{y_j^k\}}.
    \end{align*}
Observe that it converges weakly to $\mu$.

It is enough to prove that $\liminf_{k\rightarrow \infty} \PP_{\mu_k}(\ell < 1) < \frac{1}{3}$ and the result then follows by Portmanteau's Theorem. Indeed, since the set where $\ell <1$ is open it holds that $\liminf_{k\rightarrow\infty} \PP_{\mu_k}(\ell < 1) \ge \PP_\mu(\ell <1).$

Observe that
\begin{align}\label{eq:discretised}
        \PP_{\mu_k} (\ell < 1 ) = & \frac{1}{3} - 2\sum_{i=1}^n \alpha_i^2 + \frac83 \sum_{i=1}^n\alpha_i^3 + 3\left( \sum_{i=1}^n \alpha_i^2 \right)^2 - 4\sum_{i=1}^n \alpha_i^4 \\
        & - 2\sum_{j=1}^l \mu(I_j^k)^2 + \frac83 \sum_{j=1}^l\mu(I_j^k)^3 + 3\left( \sum_{j=1}^l \mu(I_j^k)^2 \right)^2 - 4\sum_{j=1}^l \mu(I_j^k)^4 \\
        & + \sum_{j=1}^l\sum_{i=1}^n\alpha_i^2\mu(I_j^k)^2.
\end{align}
We first prove that 
\begin{align}
\label{ineq3}
2\sum_{j=1}^l \mu(I_j^k)^2 - \frac83 \sum_{j=1}^l\mu(I_j^k)^3 + 4\sum_{j=1}^l \mu(I_j^k)^4 - \sum_{j=1}^l\sum_{i=1}^n\alpha_i^2\mu(I_j^k)^2 > 3\left( \sum_{j=1}^l \mu(I_j^k)^2 \right)^2
\end{align}
by implementing a similar argument as in the proof of the first case of Proposition \ref{prop1}.

From the properties of the intervals $I_j^k$, it follows that for $k$ sufficiently large it holds that
\begin{align}
\label{ine5}
    \sum_{j=1}^l \mu(I_j^k)^2 \le \frac{8p}{k} < \frac{5}{27}.
\end{align}
We first note that
\begin{align}
\sum_{j=1}^{l}\mu(I_j^k)^2 \left( 2-\sum_{i=1}^n\alpha_i^2 - \frac83 \mu(I_j^k) + 4 \mu(I_j^k)^2 \right) > \sum_{j=1}^{l}\mu(I_j^k)^2\left( 1 - \frac83 \mu(I_j^k) + 4 \mu(I_j^k)^2 \right),
\end{align}
and further that the polynomial $4x^2 -\frac83x +1$ attains its minimum of $\frac59$ for $x=\frac13.$ Thus
\begin{align*}
\sum_{j=1}^{l}\mu(I_j^k)^2\left( 2-\sum_{i=1}^n\alpha_i^2 - \frac83 \mu(I_j^k) + 4 \mu(I_j^k)^2 \right) > \frac59 \sum_{j=1}^{l}\mu(I_j^k)^2 > 3\left(\sum_{j=1}^{l}\mu(I_j^k)^2 \right)^2
\end{align*}
where the last inequality holds because of property (\ref{ine5}).

To conclude the proof, we will prove that
\begin{align}\label{eq:discrete-scaled}
 (2-C)\sum_{i=1}^n \alpha_i^2 - \frac83 \sum_{i=1}^n\alpha_i^3 - 3\left( \sum_{i=1}^n \alpha_i^2 \right)^2 + 4\sum_{i=1}^n \alpha_i^4 >0
\end{align}
and observe that the inequality does not depend on $k$. Recall that $$\sum_{i=1}^n \alpha_i = 1-p.$$

So, if we replace $\alpha_i$ by $\frac{\alpha_i}{1-p},$ since  we have $$\sum_{i=1}^n \frac{\alpha_i}{1-p} = 1,$$ Proposition \ref{prop1} implies that 
\begin{align*}
\frac{2}{(1-p)^2}\sum_{i=1}^n\alpha_i^2 - \frac{1}{(1-p)^3}\frac83 \sum_{i=1}^n \alpha_i^3 +\frac{4}{(1-p)^4}\sum_{i=1}^n \alpha_i^4 > & \frac{3}{(1-p)^4}\left( \sum_{i=1}^n \alpha_i^2 \right)^2 \cr 
 & + \frac{C}{(1-p)^2} \sum_{i=1}^n \alpha_i^2. 
\end{align*}
Rewriting it we get
\begin{align*}
(2-C)\sum_{i=1}^n\alpha_i^2 - \frac83 \sum_{i=1}^n \alpha_i^3 +4\sum_{i=1}^n \alpha_i^4 + (2p^2 -4p)\sum_{i=1}^n \alpha_i^2 +\frac83 p\sum_{i=1}^n\alpha_i^3  > 3\left( \sum_{i=1}^n \alpha_i^2 \right)^2,
\end{align*}
so it suffices to prove that
\begin{align*}
(2p -4)\sum_{i=1}^n \alpha_i^2 +\frac83 \sum_{i=1}^n\alpha_i^3 < 0.
\end{align*}
This is true, since
\begin{align*}
  \frac83  \sum_{i=1}^n\alpha_i^3 < \frac{8}{3}(1-p)\sum_{i=1}^n\alpha_i^2 \le   (4-2p)\sum_{i=1}^n\alpha_i^2,
 \end{align*}
completing the proof of the inequality. Since it does not depend on $k$, the result follows.
\end{proof}

\begin{remark}\label{rmk:continuity} Notice that, if $\alpha_i < \delta$ for $i=1,\dots,n,$ then the argument above implies that 
\begin{equation}\label{eq:almost-equality}
\mathbb{P}_{\mu}(\ell < 1) \ge \frac13 - C \delta^2.
\end{equation}
This shows a continuity-type result around continuous measures: if the discrete part of the measure $\mu$ has very small individual atoms - and hence is, in a way, close to being continuous -, then $\mu$ almost attais $A(1).$ 

On the other hand, suppose that one has the information that (\ref{eq:almost-equality}) holds for a certain measure $\mu.$ The next result shows the \emph{converse} of the statement above; that is, in case (\ref{eq:almost-equality}) holds, then $\mu$ has discrete part composed by small individual atoms. 
\end{remark}

\begin{proposition}\label{prop:stab-cont} Let $\mu$ be as in Proposition \ref{prop:only-extreme-continuous}, and suppose that (\ref{eq:almost-equality}) holds. Then we have that
\[
\sum_{p \in \mathbb{S}^1 \colon \mu(\{p\}) \neq 0} \mu(\{p\})^2 \lesssim \delta^2. 
\]
As a consequence, it holds that $\mu(\{p\}) \lesssim \delta,$ for any $p \in \mathbb{S}^1.$
\end{proposition}

\begin{proof} We employ again a discretisation argument, as in the proof of Proposition \ref{prop:only-extreme-continuous}. By doing so, after computing (\ref{eq:discretised}) and combining with (\ref{ineq3}), Portmanteau's theorem implies that 
\begin{align*}
\frac13 - C\delta^2 & \le \mathbb{P}_{\mu}(\ell < 1) \le \liminf_{k \to \infty} \mathbb{P}_{\mu_k}(\ell < 1) \cr 
 & \le \frac13 - 2\sum_{i=1}^n \alpha_i^2 + \frac83 \sum_{i=1}^n\alpha_i^3 + 3\left( \sum_{i=1}^n \alpha_i^2 \right)^2 - 4\sum_{i=1}^n \alpha_i^4. 
\end{align*}
Using now (\ref{eq:discrete-scaled}), we readily obtain that 
\[
\sum_{i=1}^n \alpha_i^2 \le C' \delta^2,
\]
for some absolute $C'>0.$ The result then follows directly from there. 
\end{proof}

\begin{remark} Note that the proof of Proposition \ref{prop:stab-cont} shows a stronger statement about the resemblance of such measures to extremisers of our original result: indeed, it shows that the probability (according to $\mu$) of drawing the same point twice is small, which is a stronger way of saying that the measure $\mu$ is close to being continuous.

Moreover, Proposition \ref{prop:stab-cont} is indeed \emph{sharp}: let $\tilde{\mu}_n$ denote a discrete measure with atoms at the vertices of a regular $n-$sided polygon inscribed in $\mathbb{S}^1,$ but now assign to each vertex except one the weight $\frac{1-\delta}{n-1},$ and to the remaining vertex assign weight $\delta>0.$ Then, by Lemma \ref{lemma1}, for $n$ sufficiently large and $\delta > 0$ sufficiently small, we have that 
\[
\left| \mathbb{P}_{\tilde{\mu}_n}(\ell < 1) - \frac{1}{3} \right| \sim \delta^2,
\]
while we do have a vertex with weight $\delta,$ and moreover
\[
\sum_{p \in \mathbb{S}^1 \colon \tilde{\mu}_n(\{p\}) \neq 0} \tilde{\mu}_n(\{p\})^2 \ge \delta^2. 
\]
\end{remark}

\subsection{Behaviour near the boundary} In this final part, we shall prove Theorem \ref{thm:almot-boundary}.

\begin{figure}[!ht]
\begin{tikzpicture}[scale=2]
\tkzDefPoints{0/0/O, 1/0/X1}
\tkzDefPoints{0/0/O, .5/0/Y1}
\tkzDefPoints{0/0/O, 0.85/0/Z1}
\tkzDefPoint(0.48,0.15){R}
\tkzDefPoint(1,0){X}
\tkzDefPoint(0,0){O}
\tkzDefPoint(-0.85,0.53){Z}
\tkzDefPoint(0.85,0.52678){H}
\tkzDefPoint(0.85,-0.52678){F}
\tkzDefPoint(1.04,-0.1){N}
\tkzDefPointsBy[rotation=center O angle 7*360/17](H){I}
\tkzDefPoint(-0.94,-0.35){J}
\tkzDefPoint(0.425,0){K}


\draw [arrows=|-|,black,line width = 2pt,domain=-31.8:31.8] plot ({cos(\x)}, {sin(\x)});

\tkzDrawCircle(O,X1)
\tkzDrawCircle[dashed](O,Z1)
\tkzDrawSegments(F,H)
\tkzDrawSegments(O,Z1)
\tkzLabelPoint[above](K){$1-\theta$}
\tkzLabelPoint[right](N){$I_{\theta}$}
\tkzDrawSegments[arrows=<->](O,Z1)

\end{tikzpicture}
\label{fig:diagonals}
\caption{An arc $I_{\theta}$ whose length is $\alpha_{\theta}.$}
\end{figure}
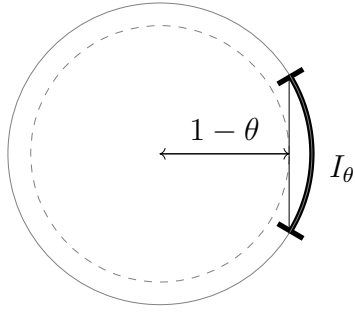

\begin{proof} We let, first of all, $\alpha_{\theta} = \arccos(1-\theta).$ In what follows, $l(J)$ denotes the (circular) length of the subset $J \subset \mathbb{S}^1.$ 

Our first observation is that, if an arc $I_{\theta}$ has length $l(I_{\theta}) = \alpha_{\theta},$ then any two chords whose endpoints lie in $I_{\theta}$ intersect \emph{outside} the disk $C_{1-\theta}.$ With that in mind, we divide into cases:

\vspace{2mm}

\noindent \textbf{Case 1:} the measure $\mu$ has continuous part - denoted by $\mu_c$ - satisfying $\mu_c(\mathbb{S}^1) \ge \frac{1}{C},$ where $C>0$ is a large numerical constant, to be determined later. 

In that case, we claim that there is an interval $I' \subset \mathbb{S}^1$ such that $l(I') \le \alpha_{\theta}$ and $\mu_c(I') \ge c \theta^{1/2},$ where $c>0$ is an absolute constant. Indeed, we may divide the circle into $\sim 1/\theta^{1/2}$ disjoint arcs of length at most $\alpha_{\theta},$ and the claim follows by pigeonholing. 

Consider then the measure $\nu = \mu_c(I')^{-1} \cdot \mu_c|_{I'}.$ This is, by definition, a probability measure defined on $\mathbb{S}^1.$ Moreover, it is, by definition, a \emph{continuous} measure. Hence, from Proposition \ref{prop:continuous-extreme}, 
\[
\mathbb{P}_{\nu}(\ell < 1) = \frac{1}{3}.
\]
That, together with the fact that $\mu_c(I') \ge c \theta^{1/2},$ implies that
\begin{equation}\label{eq:lower-bound-first}
\int_{(\mathbb{S}^1)^4} \ind_{ \{ \ell <1 \}} \d(\mu_c|_{I'}) \otimes \d(\mu_c|_{I'}) \otimes \d(\mu_c|_{I'}) \otimes \d(\mu_c|_{I'}) \ge c' \theta^2. 
\end{equation}
Since we also know that 
\[
\mathbb{P}_{\mu}(\ell < r) \le \mathbb{P}_{\mu}(\ell < 1) - \int_{(\mathbb{S}^1)^4} \ind_{ \{ \ell <1 \}} \d(\mu_c|_{I'}) \otimes \d(\mu_c|_{I'}) \otimes \d(\mu_c|_{I'}) \otimes \d(\mu_c|_{I'}),
\]
the result readily follows from (\ref{eq:lower-bound-first}) in this case. 

\vspace{2mm}

\noindent \textbf{Case 2:} the continuous part satisfies $\mu_c(\mathbb{S}^1) \le \frac{1}{C}.$  Hence, the point masses dominate over that part, and we have to ajust accordingly. We divide again into subcases. 

\vspace{2mm}

\textbf{Case 2.1:} there is a point $p_0 \in \mathbb{S}^1$ with $\mu(\{p_0\}) \ge \frac{1}{C} \cdot \theta^{1/2}.$ In that case, by using a discretisation argument similar to the one of the proof of Proposition \ref{prop:only-extreme-continuous}, together with inequality (\ref{eq:discrete-scaled}), we conclude that 
\[
\mathbb{P}_{\mu}(\ell < r) \le \mathbb{P}_{\mu}(\ell < 1) \le \frac13 - C \theta,
\]
which finishes this case as well. 

\vspace{2mm}

\textbf{Case 2.2:} One has $\mu(\{p\}) < \frac{1}{C}\cdot \theta^{1/2}$ for all $p \in \mathbb{S}^1.$ Take then $I''$ arc of $\mathbb{S}^1$ with $l(I'') \le \alpha_{\theta},$ and $\mu(I'')$ maximal. Consider the normalised measure $\tilde{\nu} = \mu(I'')^{-1} \cdot \mu|_{I''}.$ By assumption, we have $\tilde{\nu}(\{p\}) \le \frac{1}{C'},$ for all $p \in \mathbb{S}^1,$ where $C'$ depends only upon $C.$ 

Resorting again to the proof of Proposition \ref{prop:only-extreme-continuous}, we see that, if $\tilde{\nu}(\{p\}) \le \frac{1}{C'},$ then Remark \ref{rmk:continuity} implies that making $C>0$ sufficiently large, then we may make 
\[
\mathbb{P}_{\tilde{\nu}}(\ell < 1) \ge \frac{1}{4}.
\]
In particular, since $I''$ maximises $\mu$ among intervals of length at most $\alpha_{\theta},$ we have 
\[
\int_{(\mathbb{S}^1)^4} \ind_{ \{ \ell <1 \}} \d(\mu|_{I''}) \otimes \d(\mu_c|_{I''}) \otimes \d(\mu_c|_{I''}) \otimes \d(\mu_c|_{I''}) \ge \frac{1}{4} \mu(I'')^4 \gtrsim \theta^2.
\]
Then it follows at once that 
\begin{align*} 
\mathbb{P}_{\mu}(\ell < r) & \le \mathbb{P}_{\mu}(\ell < 1) - \int_{(\mathbb{S}^1)^4} \ind_{ \{ \ell <1 \}} \d(\mu|_{I''}) \otimes \d(\mu_c|_{I''}) \otimes \d(\mu_c|_{I''}) \otimes \d(\mu_c|_{I''}) \cr 
    & \le \frac13 - C \theta^2,
\end{align*} 
which concludes the proof of our result. 
\end{proof}

\begin{remark} The decay order $\theta^2$ is likely not optimal in the proof above; the reasons for the losses of that exponent in the proof above come from the fact that we only analyzed our measure within an arc of length $\sim \theta^{1/2},$ ignoring the interactions between different arcs. 

We would expect that a more careful analysis of the interactions of different arcs, together with a more thorough combinatorial argument, would yield a sharper version of Theorem \ref{thm:almot-boundary}. Furthermore, we believe the question of what should be the sharp version of the result above to be an extremely interesting one, which we would like to address in future work. 
\end{remark}

\section{Behaviour for $r<\frac12$}\label{sec:small-r}

In this section, we aim to prove Theorem \ref{rsmaller12}. For that, we start with a definition:
\begin{definition}
    Let $x \in \SS^1$ and $0<r<1$, we define
    \begin{align*}
        I_r(x) = \{ y\in \SS^1: \aff(x,y) \cap \overline{C_r} \neq \emptyset \}.
    \end{align*}
\end{definition}

    \begin{figure}[!h]

\begin{tikzpicture}[scale=2]
\tkzDefPoints{0/0/O, 1/0/X1}
\tkzDefPointsBy[rotation=center O angle 360/2](X1){X2}
\tkzDefPoints{0/0/O, 1/0/X1}
\tkzDefPoint(0,0.5){X2}
\tkzDefPoint(-0.5,0){X5}
\tkzDefPoint(-0.5,-0.87){X3}
\tkzDefPoint(-0.5,0.87){X4}
\tkzDefPoint(0,0){A}
\tkzDefPoint(-1,0){Y}

\tkzDrawCircle(O,X1)
\tkzDrawCircle(O,X2)
\tkzDrawLines(X1,X3)
\tkzDrawLines(X1,X4)
\tkzDrawArc[delta=0,line width=5,opacity=0.4](A,X4)(X3)

\tkzDrawPoints(X1)
\tkzDrawSegments(A,X5)
\tkzLabelSegments[pos=0.5,above](A,X5){$r$}
\tkzLabelPoint[right](X1){$x$}
\tkzLabelPoint[left](Y){$I_r(x)$}

\end{tikzpicture}
\caption{Arc $I_r(x)$ highlighted.}
\label{fig:Irx}

    \end{figure}
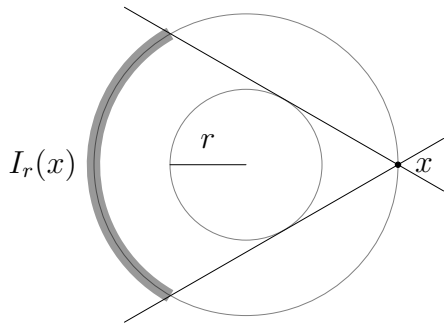

A direct consequence of the definition is that 
\begin{align}
\label{ineqintegral}
    \PP_\mu( \ell \le r) \le \left( \int_{\SS^1} \mu(I_r(x))\d\mu(x) \right)^2,
\end{align}
as two chords can only intersect in $\overline{C_r}$ if they both intersect $\overline{C_r}.$ Our goal is to explore the behaviour of the right-hand side of the inequality. In order to do that, we need a notion of what it means to locally maximize that quantity. 

\begin{definition}
\label{def1}
    We say $\mu$ is a local maxima of the quantity
    \begin{align*}
        \int_{\SS^1} \mu(I_r(x))\d\mu(x),
    \end{align*}
if there exists $\epsilon_0>0$ such that for every $0< \epsilon < \epsilon_0$ and $\nu \in \mathcal{M}(\SS^1)$ it holds that
    \begin{align}
    \label{ineq6}
        \int_{\SS^1} [(1-\epsilon)\mu(I_r(x))+\epsilon \nu(I_r(x))]\d ((1-\epsilon)\mu+\epsilon \nu) \le  \int_{\SS^1} \mu(I_r(x))\d\mu(x).
    \end{align}
\end{definition}
We prove a few lemmas concerning the structure and properties of local and glocal maximisers. 

\begin{lemma}\label{lemma:chord-distance-max}
    There exists $\mu_\star \in \mathcal{M}(\SS^1)$ extremiser of
    \begin{align}\label{eq:chord-distance-functional}
        \sup_{\mu \in \mathcal{M}\SS^1} \int_{\SS^1} \mu(I_r(x))\d\mu(x).
    \end{align}
\end{lemma}
\begin{proof}
    For $0\le r \le 1,$ let $\{\mu_n\}_{n=1}^\infty \subset \mathcal{M}(\SS^1)$ be such that
      \begin{align}
    \lim_{n\rightarrow\infty} \int_{\SS^1} \mu_n(I_r(x))\d\mu_n(x) =   \sup_{\mu \in \mathcal{M}(\SS^1)} \int_{\SS^1} \mu(I_r(x))\d\mu(x).
    \end{align}
    Since $\SS^1$ is compact, by Prokhorov's theorem, we can suppose, without loss of generality, that $\mu_n$ converges weakly to a measure $\mu_\star.$ We observe that the function $\ind_{\{y \in I_r(x)\}}$ is upper semi-continuous. So, from Portmanteau's theorem, we conclude that
    \begin{align*}
        \limsup_{n\rightarrow\infty} \int_{\SS^1}\int_{\SS^1 }\ind_{\{y \in I_r(x)\}}\d\mu_n(y)\d\mu_n(x) \le \int_{\SS^1}\int_{\SS^1 }\ind_{\{y \in I_r(x)\}}\d\mu_\star(y)\d\mu_\star(x),
    \end{align*}
    concluding the proof of the result.
\end{proof}

We observe that the existence of extermisers implies the existence of local extremisers. The next result then tells us that local maximisers must satisfy additional geometric constraints, given by other probability measures.  

\begin{lemma}
\label{lemmaineq}
    Let $\mu$ be a local maxima as in Definition \ref{def1}, then it holds that
    \begin{align}
        2\int_{\SS^1} \mu(I_r(x)) \d\mu(x) \ge \int_{\SS^1} \mu(I_r(x)) \d\nu(x)+\int_{\SS^1} \nu(I_r(x)) \d\mu(x),
    \end{align}
for any $\nu \in \mathcal{M}(\SS^1).$
\end{lemma}
\begin{proof}
Define $\nu_\epsilon = (1-\epsilon)\mu + \epsilon\nu$. We use the linearity of integral and expand inequality (\ref{ineq6})
\begin{align*}
    \int_{\SS^1}\nu_\epsilon(I_r(x))\d\nu_\epsilon(x)& = (1-\epsilon)^2\int_{\SS^1}\mu(I_r(x))\d\mu(x)  + \epsilon(1-\epsilon)\int_{\SS^1}\nu(I_r(x))\d\nu(x) \\
    & \text{    } +\epsilon\int_{\SS^1} \nu(I_r(x))\du\nu_{\epsilon}(x)\\
    & = \int_{\SS^1}\mu(I_r(x))\d\mu(x)  -2\epsilon\int_{\SS^1}\mu(I_r(x))\d\mu(x) \\
    &\text{    } +\epsilon\int_{\SS^1}\mu(I_r(x))\d\nu(x) +\epsilon(1-\epsilon)\int_{\SS^1}\nu(I_r(x))\d\mu(x) \\
    &\text{    } +\epsilon^2\int_{\SS^1}\nu(I_r(x))\d\nu(x).
\end{align*}
Dividing the inequality by $\epsilon$ and taking the limit as $\epsilon$ tends to $0$, we obtain
\begin{align*}
    0 \ge -2 \int_{\SS^1}\mu(I_r(x))\d\mu(x) + \int_{\SS^1}\nu(I_r(x))\d\mu(x) +\int_{\SS^1}\mu(I_r(x))\d\nu(x),
\end{align*}
    as we wanted.
\end{proof}

We are now ready to prove Theorem \ref{rsmaller12}. 

\begin{proof}[Proof of Theorem \ref{rsmaller12}] 
    From inequality (\ref{ineqintegral}), we observe that it suffices to prove that 
    \begin{align*}
        \int_{\SS^1}\mu(I_r(x))\d\mu(x) \le \frac12,
    \end{align*}
    for any measure $\mu \in\mathcal{M}(\SS^1).$ It is enough to prove this inequality for $\mu$ a local maxima as defined in Definition \ref{def1}, since we have proved in Lemma \ref{lemma:chord-distance-max} that global extremisers do exist. So we assume that $\mu$ is a local maximiser, and we assume by contradiction that it satisfies 
        \begin{align*}
        \int_{\SS^1}\mu(I_r(x))\d\mu(x) > \frac12.
    \end{align*}
    This implies the existence of $x_\star$ such that $\mu(I_r(x_\star)) > \frac12.$ Taking $\nu = \delta_{\{y\}}$ in Lemma \ref{lemmaineq}, we observe that
    \begin{align*}
        \int_{\SS^1} \mu(I_r(x))\d\mu(x) \ge \mu(I_r(y)),
    \end{align*}
    for $y$ $\mu$-almost everywhere, so $\mu(I_r(y))$ is constant $\mu$-almost everywhere. Since $I_r(x_\star)$ has positive $\mu-$measure, there exists a $y_\star \in I_r(x_\star)$ such that $\mu(I_r(y_\star))> \frac12.$ 
Observe that if $I_r(x_\star)$ and $I_r(y_\star)$ do not intersect we have a contradiction to the fact that $\mu$ is a probability measure. Note that if $x = e^{it},$ then  $I_r(x) = (e^{i(t + \pi -2\arcsin(r))}, e^{i(t+ \pi+2\arcsin(r))}),$ so $I_r(x_\star)$ and $I_r(y_\star)$ do not intersect whenever
\begin{align}
    \pi - 2\arcsin(r) > 4\arcsin(r),
\end{align}
    which is valid as long as $0\le r < \frac12,$ concluding the proof.
\end{proof}

Observe that the inequality (\ref{ineqintegral}) only provides a non-trivial bound up to $r=\frac12.$ Indeed, for any $r>\frac12$, we observe that a discrete measure with 3 atoms forming an equilateral triangle with equal weights to its vertices satisfies
\begin{align}\label{eq:equivalt}
    \left( \int_{\SS^1}\mu(I_r(x)) \d\mu(x) \right)^2 = \frac49,
\end{align}
which is worse than the trivial bound $\frac13.$ On the other hand, if one changes perspective and looks at the problem of determining the best measure in Lemma \ref{lemma:chord-distance-max}, that is, the extremal measure for (\ref{eq:chord-distance-functional}), the next result allows us to compute more values explicitly, showing that the bound in (\ref{eq:equivalt}) is indeed \emph{sharp} if $r$ is sufficiently close to $\frac12.$

\begin{proposition} For $r \in \left(\frac12, \sin\left(\frac{3\pi}{14}\right)\right),$ we have that 
\[
\sup_{\mu \in \mathcal{M}(\SS^1)} \int_{\SS^1} \mu(I_r(x))\d\mu(x) = \frac23.
\]
\end{proposition}

\begin{proof} We begin in a similar way to the proof of Proposition \ref{rsmaller12}. Indeed, suppose $\mu$ is a measure with 
\[
 \int_{\SS^1} \mu(I_r(x))\d\mu(x) > 1/2.
\]
By the same argument as in the proof of Proposition \ref{rsmaller12}, if $\mu$ is a local maximiser of \eqref{eq:chord-distance-functional}, then $\mu(I_r(y))$ is constant $\mu-$almost everywhere. If $x_*$ is a point with $\mu(I_r(x_*)) > \frac12,$ then there is $y_* \in I_r(x_*)$ with $\mu(I_r(y_*)) > \frac12$ as well. 

Again, if $\mu(I_r(x_*) \cap I_r(y_*)) = 0,$ then we arrive at a contradiction; so we must have that there is $z_* \in I_r(x_*) \cap I_r(y_*)$ with $\mu(I_r(z_*)) > \frac12.$  In order not to arrive at a contradiction, we must have also $$\min\{\mu(I_r(x_*) \cap I_r(z_*)), \mu(I_r(y_*) \cap I_r(z_*))\} >0.$$

For $p \in \mathbb{S}^1,$ let then $A_r(p)$ be the union of two arcs, each of which has length $6 \arcsin(r) - \pi$, one starting at one endpoint of $I_r(p),$ and the other ending at the other endpoint $I_r(p).$  We claim that 
\[
\mu\left(I_r(x_*) \setminus A_r(x_*)\right) = 0,
\]
and the same holds for $y_*$ and $z_*.$ 

\begin{center}
\begin{figure}[h!]
\begin{tikzpicture}[scale=2]
\tkzDefPoints{0/0/O, 1/0/X1}
\tkzDefPointsBy[rotation=center O angle 360/2](X1){X2}
\tkzDefPoints{0/0/O, 1/0/X1}
\tkzDefPoint(0,0.5){X2}
\tkzDefPoint(-0.5,0){X5}
\tkzDefPoint(-0.5,-0.87){X3}
\tkzDefPoint(-0.5,0.87){X4}
\tkzDefPoint(-0.7,0.714){X6}
\tkzDefPoint(-0.7,-0.714){X7}
\tkzDefPoint(0,0){A}
\tkzDefPoint(-1.1,0){Y}
\tkzDefPoint(-1.2,0.3){Y3}
\tkzDefPoint(-1.2,-0.3){Y2}

\tkzDrawCircle(O,X1)
\tkzDrawCircle(O,X2)
\tkzDrawLines(X1,X3)
\tkzDrawLines(X1,X4)
\tkzDrawLines[dashed](X7,Y2)
\tkzDrawLines[dashed](X6,Y3)
\tkzDrawArc[delta=0,line width=5,opacity=0.8](A,X4)(X6)
\tkzDrawArc[delta=0,line width=5,opacity=0.8](A,X7)(X3)

\tkzDrawPoints(X1)
\tkzLabelPoint[right](X1){$p$}
\tkzLabelPoint[left](Y){$A_r(p)$}

\end{tikzpicture}
\caption{The location of the set $A_r(p)$}
\end{figure}
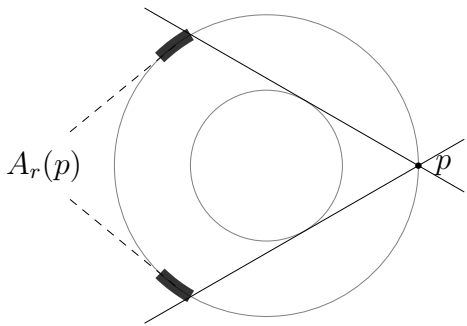
\end{center}

Indeed, if that is not the case, then there is $x' \in I_r(x_*) \setminus A_r(x_*)$ with $\mu(I_r(x')) > \frac12.$ By the definition of $A_r(p),$ it is easy to see that $I_r(x') \cap I_r(x_*) = \emptyset,$ whenever there is a point available to perform this procedure; that is, whenever 
\begin{align*} 
12 & \arcsin(r) - 2 \pi < 4 \arcsin(r) \cr 
\iff & \arcsin(r) < \frac{\pi}{4} \cr 
\iff & r < \frac{\sqrt{2}}{2},
\end{align*} 
and hence we reach a contradiction in  that case. Since we showed that 

$$\min\{\mu(I_r(x_*) \cap I_r(z_*)), \mu(I_r(y_*) \cap I_r(z_*)), \mu(I_r(x_*) \cap I_r(y_*))\} >0,$$
the considerations above imply then that 
$$\min\{\mu(A_r(x_*) \cap A_r(y_*)),\mu( A_r(y_*) \cap A_r(z_*)),\mu( A_r(z_*) \cap A_r(x_*))\} > 0.$$
We then claim that, since $A_r(x_*), A_r(y_*),A_r(z_*)$ are each a union of two arcs, the union 
$$S := A_r(x_*) \cup A_r(y_*) \cup A_r(z_*)$$
consists of at most three disjoint arcs. Indeed, let $A_1,A_2$ be the two disjoint arcs defining $A_r(x_*),$ and defined $B_1,B_2$ analogously for $A_r(y_*)$ and $C_1,C_2$ for $A_r(z_*).$ We define a graph on six vertices, where each vertex is one of the arcs $A_1,A_2,B_1,B_2,C_1,C_2,$ and two vertices are connected if and only if the respective arcs intersect. Then it follows that the graph is tripartite, and there is at least one edge between the three sets of two vertices. 

To each connected component corresponds an arc; moreover, note that the graph is triangle-free, since if, say, $A_1 \cap C_1, A_1 \cap B_1$ and $B_1 \cap C_1$ are all nonempty, 
we would have $\text{dist}(z_*,I_r(z_*)) < \frac{\pi}{2},$ which cannot happen for $r < \frac{\sqrt{2}}{2}.$ Hence, it is not difficult to conclude that the graph must have at most three distinct connected components.

Additionally, we claim that if $r < \sin\left(\frac{3\pi}{14}\right),$ then we will have exactly three of such arcs, each of which having length at most $12 \arcsin(r) - 2\pi.$ 

The proof of this claim follows from the way we constructed $I_r(x_*), I_r(y_*), I_r(z_*)$: indeed, suppose that one of the endpoints of $I_r(x_*)$ is $1.$ Then there are two arcs $A_1 \subset A_r(x_*), B_1 \subset A_r(y_*),$ such that $A_1 \cap B_1 \neq \emptyset$ and 
$$A_1 \cup B_1 \subset (e^{-i(6 \arcsin(r) - \pi)}, e^{i(6\arcsin(r)-\pi)}).$$
If we let $z_* = e^{is_0},$ then $I_r(z_*) = (e^{i(s_0 + \pi -2\arcsin(r))},e^{i(s_0 + \pi + 2\arcsin(r))}).$ Since we take $z_* \in I_r(x_*) \cap I_r(y_*),$ we have $z_* \in A_1 \cap B_1$ as long as $r \le \frac{\sqrt{2}}{2}.$ This implies that $s_0 \in (-(6\arcsin(r) - \pi),6\arcsin(r) - \pi).$ Hence, $I_r(z_*)$ is disjoint from $A_1 \cup B_1$ as long as 
\[
4 \arcsin(r) > 6 \arcsin(r) - \pi
\]
and 
\[
8 \arcsin(r) < 3 \pi - 6 \arcsin(r).
\]
The first condition is trivially true, while the second is true if, and only if, $r < \sin\left( \frac{3 \pi}{14}\right).$ Hence for such $r$ the set $I_r(z_*)$ is disjoint from the intersection $I_r(x_*) \cap I_r(y_*).$ On the other hand, with this parametrisation we have 
\[
I_r(x_*) \cup I_r(y_*) \subset \left( \mathbb{S}^1 \setminus \left(e^{4i \arcsin(r)}, e^{2i(\pi - 2\arcsin(r))}\right)\right).
\]
Since this implies that the arc-length of the complement of $I_r(x_*) \cup I_r(y_*)$ is at least 
\[
2 \pi - 8 \arcsin(r) > 6 \arcsin(r) - \pi 
\]
whenever $r < \sin\left( \frac{3\pi}{14}\right),$ then the intervals $C_1,C_2$ constituting $I_r(z_*)$ cannot intersect both $I_r(x_*)$ and $I_r(y_*).$ 
Rewrite thus
$$ S = J_1 \cup J_2 \cup J_3,$$
where $J_i$ are the arcs mentioned before, and let 
$$a_1 = \mu(J_1), a_2 = \mu(J_2), a_3 = \mu(J_3).$$
We clearly have $a_1 + a_2 + a_3 = 1$ by the considerations above. Furthermore, since the length of each arc is at most $12 \arcsin(r) - 2 \pi < 2 \arccos(r)$ for $r < \sin\left(\frac{3\pi}{14}\right),$ if we have a chord drawn from endpoints within the same $J_i,$ then the chord will not intersect the circle $C_r.$ We have hence
\[
\int_{\SS^1} \mu(I_r(x))\d\mu(x) \le 2(a_1 a_2 +a_2a_3 + a_1a_3) = 1 - (a_1^2 + a_2^2 + a_3^2) \le \frac{2}{3}. 
\]
Clearly, equality holds by simply taking three deltas, with weight $1/3$ each, at the vertices of an equilateral triangle lying on $\SS^1.$ This finishes the proof.  
\end{proof}

\begin{remark}[Expected value and energy minimisation] In the proof of Proposition \ref{rsmaller12}, we showed that one can get sharp bounds for the problem of intersecting chords by getting bounds for the quantity
\begin{align}
\label{energy}
    \int_{\SS^1} \int_{\SS^1} \ind_{\{y \in I_r(x)\}}\d\mu(y)\d\mu(x). 
\end{align}
   This problem is related to a well-known energy minimisation problem. We consider $\SS^d \subset \RR^{d+1}$ the sphere and a function $F:[-1,1] \rightarrow \mathbb{R}$. Let $\mathcal{B}$ be the set of Borel probability measures in $\SS^d$ and let $\mu \in \mathcal{B}$. We are interested in the quantity
\begin{align*}
    I_F(\mu) = \int_{\SS^d}\int_{\SS^d} F(x\cdot y) \d\mu(x)\d\mu(y).
\end{align*}
Indeed, observe that we can rephrase the quantity (\ref{energy}) as  the following: \begin{align}
    \int_{\SS^1} \int_{\SS^1} \ind_{\{y \in I_r(x)\}}\d\mu(y)\d\mu(x) = \int_{\SS^1} \int_{\SS^1} \ind_{\{ x\cdot y \le r^2-1 \}}\d\mu(y)\d\mu(x).
\end{align} 
A particularly interesting question is when is it possible to characterise optimal measures, that is, measures that minimise or maximise $I_F(\mu)$ for a fixed function $F$, or to compute extremal values of the quantity $I_F(\mu)$. Indeed, such an instance is when one is interested in computing the maximal \emph{expected value} of the distance of a chord to the origin, which is picked according to an antipodally symmetric distribution on the circle. We postpone the discussion of that to the next section.
\end{remark}

\section{Comments and remarks}\label{sec:general}

\subsection{Local optimisers and Euler-Lagrange equations} For this part, we will denote $F_r= f^{-1}([0,r))$ and, if $\nu$ any probability measure in $\mathbb{S}^1$, $G_r(\nu) = \nu \otimes \nu \otimes \nu \otimes \nu(F_r).$

A natural question arising from the definition of the functions $A(r)$ and $\overline{A}(r)$ is whether extremal probability measures exists. It turns out that we are able to provide a quick answer to this question, at least in the case of $\overline{A}:$ 

\begin{proposition}
For every $0\le r \le 1$, there is a probability measure $\nu_r$ such that $ \PP_{\nu_r}(\ell \le r) = \overline{A}(r).$
\end{proposition}

\begin{proof}
    Fix $0\le r \le 1$ and let $(\mu_k)_{k=1}^{\infty}$ be a sequence in $\mathcal{M}(\mathbb{S}^1)$ be such that
    \begin{align*}
        \lim_{k\rightarrow \infty} \PP_{\mu_k}(\ell \le r) = \overline{A}(r). 
    \end{align*}
Since $\SS^1$ is compact, by Prokhorov's theorem we can suppose, without loss of generality, that $\mu_k$ converges weakly to a measure $\nu_r.$

Using that $f$ is continuous and Portmanteau's Theorem we conclude that
\begin{align*}
  \overline{A}(r) = \lim_{k\rightarrow \infty} \mu_k \otimes \mu_k \otimes \mu_k \otimes \mu_k (\overline{F_r}) \le \nu_r \otimes \nu_r \otimes \nu_r \otimes \nu_r (\overline{F_r}) = \PP_{\nu_r}( \ell \le r),
\end{align*}
so $\nu_r$ is an extremiser, as desired. 
\end{proof}

Note that the same method cannot be used to prove that extremisers exist for $A(r).$ Indeed, for the proof above, we need to use a different version of Portmanteau's theorem, whcih would amount to proving that the limiting measure does not attribute a positive mass to the event of diagonals intersecting on a single circle - which if false if, for instance, the limiting measure is a suitably picked discrete measure. Moreover, it is easy to see that $A(r)$ and $\overline{A}(r)$ do not coincide for $r=0,1$. 

Directly related to the notion of extremality is that of local maxima. Indeed, in complete analogy to what we discussed in the case of the one-chord problem, we say $\mu$ is a \emph{local maximiser} for the problem of maximising $\mathbb{P}_{\mu}(\ell \le r)$ if there exists $\epsilon_0 > 0 $ such that, for any $0<\epsilon < \epsilon_0$ and $\nu \in \mathcal{M}(\SS^1)$, we have
    \begin{align*} 
        G_r((1-\epsilon)\mu + \epsilon \nu) \le G_r(\mu).
    \end{align*}
The concept of local maxima in those problems in which we are interested is useful for obtaining an appropriate Euler-Lagrange-type equation for extremal measures. Indeed, if $\mu$ is a local maximum, it holds that for every $\nu \in \mathcal{M}(\mathbb{S}^1)$
\begin{align}\label{eq:EL-ineq}
  \nu \otimes \mu \otimes \mu \otimes \mu ( F_r) \le   G_r(\mu).
  \end{align}
The proof of this fact follows by a direct computation: let $\nu_\epsilon = \epsilon \nu + (1-\epsilon)\mu)$ for $0\le \epsilon \le \epsilon_0$ and observe that, using the linearity of the integral, we have 
\begin{align*}
G_r(\nu_\epsilon) & = \int_{(\mathbb{S}^1)^4} \ind_{ \{ |f(x_1,x_2,x_3,x_4)|<r \}} \d\nu_\epsilon \d\nu_\epsilon\d\nu_\epsilon\d\nu_\epsilon \\
& = (1-\epsilon)^4G_r(\mu) + \epsilon^4G_r(\nu) +  4 \epsilon^3(1-\epsilon)\int_{(\mathbb{S}^1)^4} \ind_{ \{ |f(x_1,x_2,x_3,x_4)|<r \}} \d\nu\d\nu\d\nu\d\mu \\
& +  4 \epsilon(1-\epsilon)^3\int_{(\mathbb{S}^1)^4} \ind_{ \{ |f(x_1,x_2,x_3,x_4)|<r \}} \d\mu\d\mu\d\mu\d\nu \\
& + 2\epsilon^2(1-\epsilon)^2\int_{(\mathbb{S}^1)^4} \ind_{ \{ |f(x_1,x_2,x_3,x_4)|<r \}} \d\mu\d\mu\d\nu\d\nu \\
& + 4\epsilon^2(1-\epsilon)^2\int_{(\mathbb{S}^1)^4} \ind_{ \{ |f(x_1,x_2,x_3,x_4)|<r \}} \d\mu\d\nu\d\mu\d\nu \le G_r(\mu),
\end{align*}
where the last inequality holds because $\mu$ is local maxima.
Gathering the terms multiplying $G_r(\mu)$ and dividing the equation by $\epsilon$ we get
\begin{align*}
 \epsilon^3G_r(\nu) & +  4 \epsilon^2(1-\epsilon)\int_{(\mathbb{S}^1)^4} \ind_{ \{ |f(x_1,x_2,x_3,x_4)|<r \}} \d\nu\d\nu\d\nu\d\mu \\
& +  4 (1-\epsilon)^3\int_{(\mathbb{S}^1)^4} \ind_{ \{ |f(x_1,x_2,x_3,x_4)|<r \}} \d\mu\d\mu\d\mu\d\nu \\
& + 2\epsilon(1-\epsilon)^2\int_{(\mathbb{S}^1)^4} \ind_{ \{ |f(x_1,x_2,x_3,x_4)|<r \}} \d\mu\d\mu\d\nu\d\nu \\
& + 4\epsilon(1-\epsilon)^2\int_{(\mathbb{S}^1)^4} \ind_{ \{ |f(x_1,x_2,x_3,x_4)|<r \}} \d\mu\d\nu\d\mu\d\nu \le (4-6\epsilon+4\epsilon^2-\epsilon^3) G_r(\mu)
\end{align*}
Taking the limit when $\epsilon$ goes to $0$ we get \eqref{eq:EL-ineq}. A direct use of that inequality then yield the following:

\begin{corollary}
    Suppose $\mu$ is a local maximum. Then, $\mu-$almost every $y \in \mathbb{S}^1$ it holds that
    \begin{align}\label{eq:EL-mu}
        \mu\otimes \mu \otimes\mu( (y_2,y_3,y_4) \in (\mathbb{S}^1)^3 : f(y,y_2,y_3,y_4) \le r) = G_r(\mu).
    \end{align}
\end{corollary}

\begin{proof}
Let $y \in \mathbb{S}^1$ and take $\nu = \delta_{\{y\}}$ in \eqref{eq:EL-ineq}. It follows that 
\begin{equation}\label{eq:fourfold-sym}
\begin{split}
     \mu\otimes \mu & \otimes\mu( (y_2,y_3,y_4) \in (\mathbb{S}^1)^3 : f(y,y_2,y_3,y_4) \le r) \le G_r(\mu) \\
     & =\int_{\mathbb{S}^1}\mu\otimes \mu \otimes\mu( (y_2,y_3,y_4 \in (\mathbb{S}^1)^3 : f(y,y_2,y_3,y_4) \le r)\d\mu(y) ,
\end{split}
\end{equation}
so we must have equality $\mu$-almost everywhere.
\end{proof}

This last result shows that maximisers achieving $\overline{A}(r)$ have to be fairly special - indeed, they satisfy \eqref{eq:fourfold-sym}, which is a rather weak form of radial symmetry. It is expected then that the locally optimal measures, or even that the measures satisfying the Euler-Lagrange-type equation \eqref{eq:EL-mu}, have a very regular structure.  

For instance, the radial symmetry of the uniform measure on the unit circle shows that that measure satisfies \eqref{eq:EL-mu} is satisfied trivially: indeed, the radial invariance allows us to always suppose that, whenever we are drawing four points $X_1,X_2,X_3,X_4 \distas d U(\mathbb{S}^1),$ the first of them is fixed at $1$. Since $1$ plays no special role whatsoever, we could have fixed any point $y \in \mathbb{S}^1,$ and the same argument works. Hence, one arrives at the fact that the left-hand side of \eqref{eq:EL-mu} is constant for $\mu = U(\mathbb{S}^1).$

As a corollary of Theorem \ref{rsmaller12}, we have that the measures $\frac12\left( \delta_p + \delta_{-p}\right)$ all satisfy \eqref{eq:EL-mu}, for $r \in [0,1/2].$ It is curious, however, to note that those are far from being the only two-point masses satisfying this: indeed, for any chord $\ell_1$ with endpoints $p_1,p_2$ for which $\ell_1 \cap C_r \neq \emptyset,$ then $\frac12\left(\delta_{p_1} +\delta_{p_2}\right)$ is also a maximiser to the problem of finding $A(r),$ and hence also must satisfy \eqref{eq:EL-mu}. 

\subsection{Maximal expected value of midpoint of random chords} As promised at the end of Section \ref{sec:small-r}, we discuss below the related problem of finding the maximal and minimal \emph{expected values} for the simplified problem of midpoints of a random chord. 

\begin{proposition}\label{prop:bounds} Let $\mu \in \mathcal{M}(\mathbb{S}^1)$, and let $X,Y$ be independent, identically distributed random variables with $X \distas d \mu.$ Then 
\begin{align*}
\mathbb{E}_\mu(\text{\emph{dist}}(\text{\emph{aff}}(X,Y),0)) = \mathbb{E}_\mu \left( \frac{|X+Y|}{2} \right) \ge \frac{1}{2}.
\end{align*}
Equality holds if, and only if, we have $d\mu = \frac{1}{2} \left( \delta_{p} + \delta_{-p}\right),$ with $p \in \mathbb{S}^1.$ 

Furthermore, suppose now that $\mu$ is antipodally symmetric. Then we have the \emph{upper bound}
\begin{align*}
\mathbb{E}_\mu \left(\frac{|X+Y|}{2}\right) \le \frac{2}{\pi},
\end{align*}
with equality if, and only if, $ \d \mu = \frac{1}{2\pi} \cdot \d\mathcal{H}^1|_{\mathbb{S}^1}$ is the uniform measure on the circle.  
\end{proposition}

\begin{proof} We will identify $\mathbb{S}^1$ with $[-1/2,1/2],$ through the map $t \longmapsto e^{2 \pi i t}.$ With that identification, we may identify $\mu$ with its respective measure on $[-1/2,1/2].$ This, together with the fact that $\mu$ is antipodally symmetric, shows that we need to maximise 
\begin{align*}
I_0(\mu) = &\int_{-1/2}^{1/2} \int_{-1/2}^{1/2} |e^{2 \pi i t} + e^{2\pi i r}| \, \d\mu(t) \, \d\mu(r) \cr 
 &= \sqrt{2} \int_{-1/2}^{1/2} \int_{-1/2}^{1/2} \left(1+\cos(2\pi(t-r))\right)^{1/2} \, \d\mu(t) \, \d\mu(r). 
\end{align*}
Suppose then that one has 
\[
\widehat{\mu}(n) = \int_{-1/2}^{1/2} e^{-2\pi i n t} \, \d\mu(t) =: a_n.
\]
Then we have 
\begin{align*}
I_0(\mu) & = \sqrt{2} \sum_{m,n \in \mathbb{Z}} a_m \overline{a_n} \int_{-1/2}^{1/2} \int_{-1/2}^{1/2} \left(1+\cos(2\pi(t-r))\right)^{1/2} e^{2\pi i mt} e^{-2\pi i nr} \, \d t \, \d r \cr 
        & = \sqrt{2} \sum_{n,m \in \mathbb{Z}} a_m \overline{a_n} \int_{-1/2}^{1/2} \left(\int_{-1/2}^{1/2} \left(1+\cos(2\pi(t-r))\right)^{1/2} e^{- 2\pi i n r} \, \d r \right) e^{2\pi i m t} \, \d t.
\end{align*}
Note that the inner integral above is simply 
\[
e^{- 2\pi i n t} \cdot \int_{-1/2}^{1/2} (1+\cos(2\pi r))^{1/2} \cos(2\pi n r) \, \d r = (-1)^{n+1}\frac{2\sqrt{2}}{\pi \cdot(4n^2-1)} e^{- 2\pi i n t} =: c_n e^{-2\pi i n t}.
\]
Hence, we obtain that 
\begin{align*}
I_0(\mu) & = \sqrt{2} \sum_{m, n \in \mathbb{Z}} a_m \overline{a_n} c_n \int_{-1/2}^{1/2} e^{2\pi i m t} e^{-2\pi i n t} \, \d t = \sqrt{2} \sum_{m \in \mathbb{Z}} c_m |a_m|^2 \cr 
 & \ge \sqrt{2} c_0 - \sqrt{2} \sum_{m \text{ even}} |c_m| |a_m|^2 \ge  \frac{4}{\pi} - \frac{4}{\pi} \sum_{m \text{ even}} \frac{1}{4m^2-1} \cr 
 & = \frac{4}{\pi} - \frac{8}{\pi}\cdot \left( \frac{4-\pi}{8}\right) = 1.
\end{align*}
Notice that we have equality if, and only if, we have $a_n = 0$ for each $n$ odd, and $a_n = \pm 1$ for each $n$ even. Since $|a_2| \le 1$ with equality if and only if $\mu$ is a sum of two deltas supported on antipodal points, we are restricted to that case; It is then easy to check that $a_3 = 0$ if and only if $\mu = \frac{1}{2} \left( \delta_p + \delta_{-p}\right), \, p \in \mathbb{S}^1.$ This finishes the proof of the first assertion. 

In order to prove the second one, we employ a similar tactic: indeed, we have that, since $\mu$ is antipodally symmetric, we may reduce matters to estimating 
\begin{align*}
I_0(\mu) = &\int_{-1/2}^{1/2} \int_{-1/2}^{1/2} |e^{2 \pi i t} - e^{2\pi i r}| \, \d\mu(t) \, \d\mu(r) \cr 
 &= \sqrt{2} \int_{-1/2}^{1/2} \int_{-1/2}^{1/2} \left(1-\cos(2\pi(t-r))\right)^{1/2} \, \d\mu(t) \, \d\mu(r) \cr 
 & = \sqrt{2} \sum_{m,n \in \mathbb{Z}} a_m \overline{a_n} \int_{-1/2}^{1/2} \int_{-1/2}^{1/2} \left(1-\cos(2\pi(t-r))\right)^{1/2} e^{2\pi i mt} e^{-2\pi i nr} \, \d t \, \d r \cr 
        & = \sqrt{2} \sum_{n,m \in \mathbb{Z}} a_m \overline{a_n} \int_{-1/2}^{1/2} \left(\int_{-1/2}^{1/2} \left(1-\cos(2\pi(t-r))\right)^{1/2} e^{- 2\pi i n r} \, \d r \right) e^{2\pi i m t} \, \d t. 
\end{align*} 
Note once more that we have an explicit formula for the  integral above:  
\[
e^{- 2\pi i n t} \cdot \int_{-1/2}^{1/2} (1-\cos(2\pi r))^{1/2} \cos(2\pi n r) \, \d r = -\frac{2\sqrt{2}}{\pi \cdot(4n^2-1)} e^{- 2\pi i n t} =: c_n e^{-2\pi i n t}.
\]
Hence, we obtain that 
\begin{align*}
I_0(\mu)  = \sqrt{2} \sum_{m, n \in \mathbb{Z}} a_m \overline{a_n} c_n \int_{-1/2}^{1/2} e^{2\pi i m t} e^{-2\pi i n t} \, \d t = \sqrt{2} \sum_{m \in \mathbb{Z}} c_n |a_n|^2 \le \sqrt{2} c_0 |a_0|^2 \le \frac{4}{\pi}.
\end{align*}
Notice that we have equality if, and only if, we have $a_n \equiv 0$ for each $n \neq 0.$ This can only happen if $\mu$ is \emph{exactly} the uniform measure on $\mathbb{S}^1.$ This finishes the proof of the second assertion, and hence that of the Proposition. 
\end{proof}

The first proof of the second assertion in Proposition \ref{prop:bounds} seems to be the one from \cite{Bjorck}, where the author uses tools from potential analysis in order to conclude that the only distribution which maximises $\mathbb{E}|X-Y|$ is the uniform one. The Fourier analysis approach to these kinds of problems, as employed above, has on the other hand been made popular through several contributions in the field. Indeed, we refer the reader to the works \cite{bilik}, \cite{Alexander}, and the references therein for further instances of such an use of an orthogonal basis decomposition in this setting, and to the work \cite{Foschi} for an application of such techniques to the sharp Fourier restriction problem. 

  In general, in such energy minimization problems, the function $F$ is supposed to possess certain special properties allowing one to compute - or estimate - $I_F(\mu).$ In particular, as mentioned before, in many of the results in the literature, some positivity (or precise sign alternating structure, as in the lower bound of Proposition \ref{prop:bounds}) of orthonormal expansion coefficients - such as Fourier series, spherical harmonics, orthogonal polynomials and so on - of $F$ is assumed. In our case, however, since $F$ is the characteristic of an interval, any of the desired positivity conditions above can be shown to fail. 
   Indeed, the Fourier coefficients of the function $F_r(t) = \ind_{\{ \cos(2 \pi t) \le r-1\}}$ can be shown to be
\begin{align*}
    \widehat{F_r}(k) = -\frac1{\pi k} \sin(k \arccos{(r-1)}),
\end{align*}
which do not possess the desired sign structure.

In general, we believe that the results above represent a new paradigm in the landscape of energy minimization problems - especially since the maximisers are, in a concrete sense, very non-unique. In particular, we would expect that the techniques developed above might be useful for further problems of energy minimization with a geometric flavour. 

\section*{Acknowledgements} 

The authors are thankful to E. Kowalski for encouraging them to investigate the problem of analysing the intersection of chords on the circle through the lenses of extremisation problems. We would also like to thank E. Carneiro for interesting discussions held during the preparation of this manuscript. 





 \end{document}